\documentclass[12pt,oneside,openany,article]{memoir}
\usepackage{cmap}

\nouppercaseheads
\usepackage[amsthm,thmmarks]{ntheorem}
\usepackage[noTeX]{mmap}
\usepackage{etex}

\usepackage[russian, english]{babel}
\usepackage[leqno]{mathtools}

\usepackage{mathrsfs}
\usepackage{upgreek}
\usepackage[compress,square,comma,numbers]{natbib}
\usepackage{hypernat}

\usepackage{eso-pic}
\usepackage{sfmath}
\usepackage{tikz}

\usepackage{todonotes}

\usepackage{subfig}

\usepackage{microtype}

\usepackage{array}

\setsecheadstyle{\large\bfseries\memRTLraggedright}
\setsubsecheadstyle{\bfseries\memRTLraggedright}

\usepackage{graphicx,xcolor}
\definecolor{refkey}{rgb}{0,0,1}
\definecolor{labelkey}{rgb}{1,0,0}

\usepackage{hyperxmp}
\usepackage{xr-hyper}
\usepackage{nameref}
\usepackage[pdftex,bookmarks,pdfnewwindow,plainpages=false,unicode]{hyperref}

\usepackage{bookmark, url}

\usepackage{enumitem}
\usepackage{amssymb}

\newenvironment{claim}[1][{\textup{(\theequation)}}]{\refstepcounter{equation}\vglue10pt
\begin{trivlist}
\item[{\hskip\labelsep#1}]}{\vglue10pt\end{trivlist}}

\hypersetup{
colorlinks=true,
linkcolor=black,
citecolor=black,
urlcolor=blue,
pdfauthor={Victor Ivrii},
pdftitle={Complete Semiclassical Spectral Asymptotics for Periodic and Almost Periodic Perturbations of Constant Operators},
pdfsubject={Sharp Spectral Asymptotics},
pdfkeywords={Microlocal Analysis, sharp  spectral asymptotics, integrated density of states, periodic and almost periodic operators, Diophantine conditions},
bookmarksdepth={4}
}

\theoremstyle{plain}
\newtheorem{theorem}{Theorem}[chapter]

\newtheorem{proposition}[theorem]{Proposition}
\newtheorem{corollary}[theorem]{Corollary}

\theoremstyle{definition}

\theoremstyle{remark}
\newtheorem{remark}[theorem]{Remark}

\newtheorem{conjecture}[theorem]{Conjecture}

\makeatletter
\newtheoremstyle{plainfoot}%
 {\item[\hskip\labelsep \theorem@headerfont ##1\ ##2\,\footnotemark\theorem@separator]}%
 {\item[\hskip\labelsep \theorem@headerfont ##1\ ##2\ (##3)\, \footnotemark\theorem@separator]}
\makeatother

\theoremstyle{plainfoot}
\theoremseparator{.}
\newtheorem{theorem-foot}[theorem]{Theorem}
\newtheorem{lemma-foot}[theorem]{Lemma}
\newtheorem{proposition-foot}[theorem]{Proposition}
\newtheorem{corollary-foot}[theorem]{Corollary}
\newtheorem{conjecture-foot}[theorem]{Conjecture}
\newtheorem{condition-foot}[condition]{Condition}

\theoremstyle{plainfoot}
\theoremseparator{.}
\theorembodyfont{}
\newtheorem{definition-foot}[theorem]{Definition}
\newtheorem{Problem-foot}[theorem]{Problem}

\theoremstyle{plainfoot}
\theoremseparator{.}
\theorembodyfont{}
\theoremheaderfont{\slshape}
\newtheorem{remark-foot}[theorem]{Remark}     
\newtheorem{example-foot}[theorem]{Example}
\newtheorem{problem-foot}[theorem]{Problem}

\numberwithin{equation}{chapter}

\newcommand{\bR}{\mathbb{R}}
\newcommand{\bS}{\mathbb{S}}

\newcommand{\cK}{\mathcal{K}}

\newcommand{\cV}{\mathcal{V}}
\newcommand{\cW}{\mathcal{W}}

\newcommand\sC{\mathscr{C}}
\newcommand\sL{\mathscr{L}}
\newcommand\sH{\mathscr{H}}

\newcommand{\supp}{\operatorname{supp}}

\newcommand{\Ker}{\operatorname{Ker}}

\newcommand{\N}{\mathsf{N}}

\renewcommand{\Re}{\operatorname{Re}}

\addtopsmarks{headings}{}{%
\createmark{chapter}{right}{shownumber}{}{. \ }
}
\title{Complete Differentiable Semiclassical Spectral Asymptotics
}

\author{Victor Ivrii\thanks{This research was supported in part by National Science and Engineering Research Council (Canada) Discovery Grant RGPIN 13827}}

\begin{document}

\maketitle

\begin{abstract}
For an operator $A\coloneqq A_h= A^0(hD) +  V(x,hD),$ with a ``potential'' $V$ decaying as $|x|\to \infty$ we establish under certain assumptions the complete and differentiable with respect to $\tau$ asymptotics of $e_h(x,x,\tau)$ where $e_h(x,y,\tau)$ is the Schwartz kernel of the spectral projector.
\end{abstract}

\chapter{Introduction}
\label{sect-1}

Consider a self-adjoint matrix operator
\begin{gather}
A\coloneqq A_h= A^0(hD) +  V(x,hD),
\label{eqn-1.1}\\
\shortintertext{where}
| D^\beta_\xi A^0(\xi)|\le c_{\alpha\beta}(|\xi|+1)^m \qquad \forall \beta  \ \forall \xi
\label{eqn-1.2}\\
\shortintertext{and}
A^0(\xi)\ge c_0 |\xi|^m - C_0 \qquad \qquad \ \forall \xi.
\label{eqn-1.3}
\end{gather}
We assume that $A^0(\xi)$ is \emph{$\xi$-microhyperbolic at energy level $\lambda$\/}, i.e. for each $\xi$ there exists a direction $\ell(\xi)$ such that $|\ell(\xi)|\le 1$ and 
\begin{equation}
(\langle \ell(\xi), \nabla_\xi\rangle  A^0(\xi)v,v) + |(A^0(\xi)-\lambda)v| \ge \epsilon_0 |v |^2 \qquad \forall v.
\label{eqn-1.4}
\end{equation}

Further, we assume that $V(x,\xi)$ is a real-valued function, satisfying
\begin{gather}
|D^\alpha _\xi D_x^\beta V(x,\xi)|\le  c_{\alpha\beta}(|\xi|+1)^m (|x|+1)^{-\delta-|\beta|} \qquad \forall \alpha, \beta\ \forall x, \xi
\label{eqn-1.5}\\
\shortintertext{and}
|D^\alpha _\xi D_x^\beta V(x,\xi)|\le  \varepsilon \qquad \forall \alpha, \beta\colon |\alpha|+|\beta|\le 1\ \forall x, \xi.
\label{eqn-1.6}
\end{gather}

Our main theorem is

\begin{theorem}\label{thm-1.1}
Let conditions \textup{(\ref{eqn-1.2})}-- \textup{(\ref{eqn-1.4})} and \textup{(\ref{eqn-1.6})} with sufficiently small constant $\varepsilon>0$ be fulfilled.  Then
 
\begin{enumerate}[label=(\roman*), wide, labelindent=0pt]
\item\label{thm-1.1-i}
The complete spectral asymptotics holds for $\tau \colon |\tau-\lambda|\le \epsilon$:
\begin{equation}
e_h(x,x,\tau) \sim \sum_{n\ge 0}\kappa_n (x,\tau)h^{-d+n}
\label{eqn-1.7} 
\end{equation}
where $e_h(x,y,\tau)$ is the Schwartz kernel of the spectral projector $\uptheta (\tau -A_h)$ of $A_h$.

\item\label{thm-1.1-ii} 
This asymptotics is \underline{infinitely differentiable} with respect to $\tau$.
\end{enumerate} 
\end{theorem}

\begin{remark}\label{rem-1.2}
\begin{enumerate}[label=(\roman*), wide, labelindent=0pt]
\item\label{rem-1.2-i}
Statement \ref{thm-1.1-i} was sketched under much more restrictive assumptions in Theorem~\ref{IDS-thm-3.2} of \cite{IvrIDS}; however we provide here more detailed exposition.

\item\label{rem-1.2-ii}
In Theorem~\ref{thm-2.8}  we provide the dependence of the remainder on $|x|$.

\item\label{rem-1.2-iii} 
This asymptotics is also infinitely differentiable with respect to $x$ but it is really easy. 
\end{enumerate}
\end{remark}

Differentiability and completeness of the spectral asymptotics are really different. F.e. for operators with almost periodic with respect to $x$ perturbation $V(x,hD)$ the spectral asymptotics are complete (see \cite{IvrIDS} and references there) but in dimension $1$ it is not necessarily differentiable even once due to spectral gaps. Furthermore, if we perturb an operator we study in this paper by an appropriate ``negligible'' operator (i. e. with $O(h^\infty )$ norm), the absolutely continuous spectrum on the segment $[\lambda_- ,\lambda_+]$ with $\lambda_\mp =\lambda + O(h^\infty)$ will be replaced by an eigenvalue of the infinite multiplicity and then the spectral asymptotics will complete albeit  non-differentiable even once.

To establish spectral asymptotics we apply the ``hyperbolic operator method''; namely, let us consider the Schwartz kernel of the propagator $e^{ih^{-1}tA_h}$:
\begin{equation}
u\coloneqq u_h(x,y,t)= \int e^{ih^{-1}t\tau }\,d_\tau e_h(x,y,\tau).
\label{eqn-1.8}
\end{equation}
Then under ellipticity and microhyperbolicity conditions (\ref{eqn-1.3}) and (\ref{eqn-1.4})
\begin{equation}
F_{t\to h^{-1} \tau} \bar{\chi}_T(t) u_h(x,x,t) \sim \sum _{n\ge 0}\kappa' _n(x,\tau) h^{1-d+n} ,
\label{eqn-1.9}
\end{equation}
where here and below  $\chi \in \sC_0^\infty ([-1,-\frac{1}{2}]\cup [\frac{1}{2},  1])$, 
$\bar{\chi}\in \sC_0^\infty ([-1,1])$, $\bar{\chi}(t)=1$ on $[-\frac{1}{2},\frac{1}{2}]$,
 $\chi_T(t)=\chi(t/T)$ etc,  $\kappa'_n(x,\tau)=\partial_\tau \kappa_n(x,\tau)$, and  $ T=T_*>0$ is a small constant here.

Then, due to Tauberian theorem we arrive to the spectral asymptotics with the remainder estimate $O(h^{1-d})$. Next, under different assumptions one, using propagation of singularities technique,  can prove that
\begin{equation}
|F_{t\to h^{-1} \tau} \chi _T(t) u_h(x,x,t)| = O(h^\infty)
\label{eqn-1.10}
\end{equation}
for all $T\in [T_*,T^*]$.  Then (\ref{eqn-1.9}) holds with $T=T^*$ and again, due to Tauberian theorem, we arrive to the spectral asymptotics with the remainder estimate $O(T^{*\,-1}h^{1-d})$ (provided $T^*=O(h^{-K})$). In particular, if (\ref{eqn-1.10}) holds provided  $T^*=h^{-\infty}$, we arrive to complete spectral asymptotics. This happens f.e. in the framework of \cite{IvrIDS}.

However we do not have Tauberian theorems for the derivatives (with respect to $\tau$) and we need to use an inverse Fourier transform and its derivatives
\begin{equation}
\partial ^n _\tau e_h(x,x,\tau)= 
(2\pi h)^{-1} \int_{\bR} e^{-ih^{-1} \tau t} (-ih^{-1}t)^{n-1}u_h(x,x,t)\,dt
\label{eqn-1.11}
\end{equation}  
for $n\ge 1$. If we insert a factor $\bar{\chi}_T(t)$ into integral, we will get exactly $n$-th derivative of the right-hand expression of (\ref{eqn-1.7}). However we need to estimate the remainder
\begin{equation}
(2\pi h)^{-1} \int_{\bR} e^{-ih^{-1} \tau t} (-ih^{-1}t)^{n-1} (1-\bar{\chi}_T(t)) u_h(x,x,t)\,dt
\label{eqn-1.12}
\end{equation} 
and to do this we need to properly estimate the left-hand expression of (\ref{eqn-1.10}) for all  $T\ge T_*$ (rather than for $T\in [T_*,T^*]$).

To achieve this we will use a more subtle propagation technique and prove that for $T\ge T_* (R)$ the left-hand expression of (\ref{eqn-1.10}) is $O( (h/T)^\infty)$, provided $|x|\le R$.

\chapter{Proofs}
\label{sect-2}

\section{Preliminary remarks}
\label{sect-2.1}

Observe that, due to assumptions (\ref{eqn-1.6}) and (\ref{eqn-1.5}) a propagation speed with respect to $\xi$ does not exceed $\min\bigl(\varepsilon,\, C (|x|+1)^{-1-\delta}\bigr)$ and  one can prove easily, that 
for  for a generalized Hamiltonian trajectory\footnote{\label{foot-1} For a definition of the generalized Hamiltonian trajectory see Definition \ref{monsterbook-def-2-2-8} of \cite{IvrMB}.} $(x(t),\xi(t))$ on energy level $\tau\le c$
\begin{equation}
\Sigma_\tau \coloneqq \{ (x,\xi)\colon \Ker (A(x,\xi)-\tau)\ne \{0\}\}
\label{eqn-2.1}
\end{equation}
with $A(x,\xi)=A^0(\xi) +V (x,\xi)$ we have $|\xi(t)-\xi(0)|\le \varepsilon'$ for all $t$ with $\varepsilon'=\varepsilon'(\varepsilon)\to 0$ as $\varepsilon\to 0$ and therefore

\begin{claim}\label{eqn-2.2}
Let conditions (\ref{eqn-1.2})--(\ref{eqn-1.4}), (\ref{eqn-1.5}) and (\ref{eqn-1.6}) with  $\varepsilon=\varepsilon (\varepsilon') >0$ with arbitrarily small $\varepsilon'$ be fulfilled. Then for a generalized Hamiltonian trajectory $(x(t),\xi(t))$ on $\Sigma_\tau$ 
\begin{equation}
|\xi(t)-\xi(0)|\le \varepsilon' \quad\text{and }\quad
|x(t)-x(0)|\ge \epsilon_2 |t| \qquad \forall t\in \bR.
\label{eqn-2.3}
\end{equation} 
\end{claim}

\vskip-10pt
Then we conclude immediately that inequality
\begin{equation}
|F_{t\to h^{-1}\tau} \chi_T(t) u_h(x,x,t) |\le C'_{s}(T) h^s 
\label{eqn-2.4}
\end{equation} 
holds for arbitrarily constant $T>0$. 

Combining with (\ref{eqn-1.9}) for small constant $T$ we conclude that 

\begin{claim}\label{eqn-2.5}
Let conditions (\ref{eqn-1.2})--(\ref{eqn-1.4}), (\ref{eqn-1.5}) and (\ref{eqn-1.6}) with  sufficiently small constant $\varepsilon >0$  be fulfilled. Then asymptotic decomposition (\ref{eqn-1.9}) holds with an arbitrarily large constant $T$.
\end{claim}

\section{Propagation and local energy decay}
\label{sect-2.2}

First we have the finite speed with respect to $x$ propagation:

\begin{proposition}\label{prop-2.1}
For $\tau\le c$ the following estimate holds
\begin{multline}
|F_{t\to h^{-1}\tau} \Bigl( \chi_T(t) u(x,y,t) \Bigr)| \le C'_{s} h^s R^{-s}\\
 \forall x,y\colon |x-y| \ge C_0 T, |x|+|y|\asymp  R.  
 \label{eqn-2.6}
\end{multline}
\end{proposition}

\begin{proof}
In the zone $\{x\colon |x|\asymp R\}$ we can apply scaling $x\mapsto xR^{-1}$, $t\mapsto t R^{-1}$, $h\mapsto h R^{-1}$ and apply the standard theory of Chapter~\ref{monsterbook-sect-2} of \cite{IvrMB}. The rest is trivial.
\end{proof}

Next, we consider $R\le \epsilon_1 T$ and apply energy estimate method to prove the local energy decay.
Observe that one can select smooth $\ell(\xi)$ in condition (\ref{eqn-1.4}). Consider operator $L^0(x,hF)$ with Weyl symbol $-\langle x,\ell(\xi)\rangle$ and $L(x,hD; t)=L^0+\varepsilon t$.

Then
\begin{multline}
2h^{-1}\Re i( (hD_t -A)v, Lv)_{\Omega_T}=  \\
\begin{aligned}
&(Lv,v)\bigr|_{t=0}^{t=T}-\Re i h^{-1}([hD_t-A, L] v, v)_{\Omega_T} = \\
&(Lv,v)\bigr|_{t=0}^{t=T} -\varepsilon \|v\|_{\Omega_T}^2 +\Re (ih^{-1}[A,L]v,v)_{\Omega_T},
\end{aligned}
\label{eqn-2.7}
\end{multline}
where $\|.\|_\Omega$ and $(.,.)_\Omega$ are a norm and an inner product in $\sL^2(\Omega)$ with
$\Omega=\Omega_T=\bR^d \times [0,T]\ni (x,t)$. Indeed, writing the left-hand expression as
\begin{align*}
&ih^{-1}\bigl[( (hD_t -A)v, Lv)_{\Omega_T} -  ( L, (hD_t -A)v)_{\Omega_T}\bigr]=\\
&ih^{-1}\bigl[( L(hD_t -A)v, v)_{\Omega_T} -  ( (hD_t -A) L, v)_{\Omega_T}\bigr] + (Lv,v)\bigr|_{t=0}^{t=T}
\end{align*}
because $L^*=L$, we arrive to (\ref{eqn-2.7}).

In virtue of (\ref{eqn-1.5}) and (\ref{eqn-1.6}) for sufficiently small constant $\varepsilon$ and for 
$h\le h_0(\varepsilon_1)$  the operator norm of $h^{-1}[V,L]$ from $\sH^m (\bR^d)$ to $\sL^2(\bR^d)$ does not exceed $\varepsilon_1$ with 
$\varepsilon_1=\varepsilon_1(\varepsilon)\to 0$  as $\varepsilon\to 0$, and then due to the microhyperbolicity assumption we conclude that
\begin{equation}
\Re (ih^{-1}[A,L]v,v) \ge (\epsilon_0 -2\varepsilon_1)\|v\|^2 - C\|(A-\tau)v\|^2
\label{eqn-2.8}
\end{equation}
for both $\bR^d$ and $\Omega_T$. 

Let us plug  into (\ref{eqn-2.7})  $v= \varphi_{\varepsilon} (A -\tau) e^{ih tA}w$ where $\varphi\in \sC_0^\infty ([-1,1])$, $0\le \varphi \le 1$; then for sufficiently small constant $\varepsilon>0$ we arrive to
\begin{equation}
\epsilon \|v\|_{\Omega_T}^2 + (Lv,v)\bigr|_{t=T}\le (Lv,v)\bigr|_{t=0}.
\label{eqn-2.9}
\end{equation}

On the other hand, 
\begin{gather}
\Re (Lv,v)|_{t=T} \ge \varepsilon T \|v\|^2 - C\||x|^{\frac{1}{2}}v\|^2
\label{eqn-2.10}\\
\shortintertext{with}
\||x|^{\frac{1}{2}}v\|^2= \||x|^{\frac{1}{2}}v\|_{B(0,R)}^2 + \||x|^{\frac{1}{2}}v\|_{B(0,R') \setminus B(0,R)}^2 +
\||x|^{\frac{1}{2}}v\|_{\bR^d\setminus B(0,R')}^2
\notag
\end{gather}
with $R'=C_0T$ and therefore for $R\le \varepsilon T $ 
\begin{equation}
\Re (Lv,v)|_{t=T} \ge \varepsilon T \|v\|^2 - CR' \|v\|^2 _{B(0,R')\setminus B(0,R)}  - 
\||x|^{\frac{1}{2}}v\|^2 _{\bR^d\setminus B(0,R')}.
\label{eqn-2.11}
\end{equation}
Observe that
\begin{equation*}
|(Lv,v)|_{t=0}\le C\Bigl(\|v_0\|^2 + \| |x|^{\frac{1}{2}}v_0\|^2\Bigr)
\end{equation*}
and then (\ref{eqn-2.9}) and (\ref{eqn-2.10}) imply that if $R\le \varepsilon T$ then 
\begin{equation}
\| v\|_{B(0,r)}^2 \le \sigma \|v_0\|^2 + C T^{-1}\Bigl(  \||x|^{\frac{1}{2}}v\|_{\bR^d \setminus B(0,R')}^2 + 
\|v_0\|^2 + \| |x|^{\frac{1}{2}}v_0\|^2\Bigr) 
\label{eqn-2.12}
\end{equation}
with $\sigma <1$ and $v|_0=v|_{t=0}$; recall that $\|v\|=\|v_0\|$. 

Recall that $v= e^{ih^{-1}tA}\varphi_\varepsilon (A-\tau) \psi_R(x) w$ where we plugged $\psi_R w$ instead of $w$,
$\psi\in \sC_0^\infty (B(0,1))$, $0\le \psi \le 1$  and $\psi=1$ in $B(0,\frac{1}{2})$. 

One can prove easily that $Q= \varphi_{\varepsilon} (A -\tau)$ is an operator with Weyl symbol $Q(x,\xi)$, satisfying 
\begin{equation*}
|D^\alpha_\xi D^\beta _x Q|\le C_{\alpha\beta}\varepsilon^{-|\alpha|-|\beta|} (|x|+1)^{-|\beta|}.
\end{equation*}
Then 
$\|v_0\|_{\bR^d\setminus B(0,2R)}\le C(h/R)^s\|w\|$ and therefore
$\| |x|^{\frac{1}{2}}v_0\|^2 \le 2R \|w\|^2$.  Further, then Proposition~\ref{prop-2.1} implies that  $\||x|^{\frac{1}{2}}v\|_{\bR^d \setminus B(0,R')}^2\le C(h/T)^s \|w\|$ provided or $R'\ge C_0T$
with sufficiently large $C_0$ and we arrive to

\begin{proposition}\label{prop-2.2}
In the framework of Theorem~\ref{thm-1.1}
\begin{equation}
\| \psi_R  e^{ih^{-1}TA} \varphi _\varepsilon(A-\lambda) \psi_R \|  <1
\label{eqn-2.13}
\end{equation}
provided $\varepsilon T \ge R\ge 1$. 
\end{proposition}

While this statement looks weak, it will lead to much stronger one:

\begin{proposition}\label{prop-2.3}
In the framework of Theorem~\ref{thm-1.1}
\begin{equation}
\| \psi_R  e^{ih^{-1}TA} \varphi _\varepsilon(A-\lambda) \psi_R \|  \le C_s  R^s T^{-s}
\label{eqn-2.14}
\end{equation}
provided $T\ge C_0 R $, $R\ge 1$. 
\end{proposition}

\begin{proof}
We want to prove by induction that
\begin{equation}
\| \psi_R  e^{in h^{-1}TA} \varphi _\varepsilon(A-\lambda) \psi_R \|  \le C\nu^N + C_s  n R^s T^{-s}
\label{eqn-2.15}
\end{equation}
with $\nu <1$. 

Assuming that for $n$ we have (\ref{eqn-2.15}), we apply the previous arguments on the interval $[nt, (n+1)T]$ to
$v =    e^{ih^{-1}tA}\varphi _\varepsilon(A-\lambda)\psi_{R/2} e^{ih^{-1}n TA}  \psi_R w$ and derive an estimate 
\begin{multline}
\| \psi_R e^{ih^{-1}tA}\varphi _\varepsilon(A-\lambda)\psi_{R/2} e^{ih^{-1}n TA}  \psi_R w \| \le \\
\nu  \| \varphi _\varepsilon(A-\lambda)\psi_{R} e^{ih^{-1}n TA}  \psi_R w\| + C'_K (h/R)^K\|w\|
\label{eqn-2.16}
\end{multline}
with $\nu <1$.

To make a step of induction we weed to estimate the norm of
\begin{multline}
\psi_R e^{ih^{-1}tA}\varphi _\varepsilon(A-\lambda)(1-\psi_{R/2}) e^{ih^{-1}n TA}  \psi_R w =\\
\begin{aligned}
&\psi_R e^{ih^{-1}tA}\varphi _\varepsilon(A-\lambda)(1-\psi_{R/2}) Q^+ e^{ih^{-1}n TA}  \psi_R+\\
&\psi_R e^{ih^{-1}tA}\varphi _\varepsilon(A-\lambda)(1-\psi_{R/2}) Q^-e^{ih^{-1}n TA}  \psi_R,
\end{aligned} 
\label{eqn-2.17}
\end{multline}
with $Q^\pm = Q^\pm (x,hD)$, $Q^+ + Q^-=I$ to be selected to ensure that  

\begin{claim}\label{eqn-2.18}
Generalized Hamiltonian trajectories on $\Sigma_\tau$, starting  as $t=0$ from  $\supp(Q^\pm) \cap \supp (1-\psi_{R/2})$ in the positive (negative) time direction, remain in the zone
$\{|x|\ge \epsilon_1 R+\epsilon_2 |t| \}$. 
\end{claim}

Then we show that
\begin{align}
&\|\psi_R e^{ih^{-1}tA}\varphi _\varepsilon(A-\lambda)(1-\psi_{R/2}) Q^+\|\le C_s(h/R)^s
\label{eqn-2.19}\\
\shortintertext{and}
& \|\varphi _\varepsilon(A-\lambda)(1-\psi_{R/2}) Q^-e^{ih^{-1}n TA}  \psi_R\| \le C_s(h/R)^s.
\label{eqn-2.20}
\end{align}

To achieve that consider $A^0(\xi)$ and for each $\xi$ in the narrow vicinity $\cW$ of $\Sigma_\tau$ let 
$K^+(\xi)\subset \bR^d_x\times \bar{\bR}^+$ be a forward propagation cone and $K^-(\xi)=-K^+(\xi)$ be a backward propagation cone. Let
\begin{equation}
\Omega ^\pm = \{(x,\xi) \colon x \notin  \uppi_x K^\pm (\xi) \}
\label{eqn-2.21}
\end{equation}
where $\uppi_x$ is $x$-projection.

Then $\Omega^\pm$ are open sets and since $\uppi_x K^+(\xi)\cap \uppi_x K^-(\xi)=\{0\}$ we conclude that
$\Omega^+ \cup \Omega ^-\supset \bS^{d-1}\times \cW$. We can then find smooth positively homogeneous of degree $0$ with respect to $x$ symbols $q^\pm (x,\xi)$ supported in $\Omega^\pm $ such that $q^+ + q^-=1$ on $\bS^{d-1}\times \cW$. Let $q^0 =1-(q^++q^-)$; then $(A^0-\tau)$ is elliptic on $\supp(q^0)$.

Finally, let $Q^\pm $ and $Q^0$ be operators with the symbols 
$q^\pm (x ,\xi) \phi _R(x) $ and  $q^0 (x |,\xi) \phi _R(x)$ correspondingly, where $\phi \in \sC_0^\infty (\bR^d\setminus 0)$, equal $1$ as $c^{-1}\le |x|\le c$ with large enough constant $c$. Then
\begin{equation}
Q^+ + Q^- +Q^0 =\phi _R(x),
\label{eqn-2.22}
\end{equation}
where (\ref{eqn-2.18}) holds and $(A-\tau)$ is elliptic on the support of the symbol of $Q^0$.

Then Proposition~\ref{prop-2.4} below implies that for $R\le \varepsilon T$ with sufficiently small constant $\varepsilon$ both (\ref{eqn-2.19}) and (\ref{eqn-2.20}) hold. On the other hand, ellipticity of $(A-\tau)$ on 
$\supp (Q^0)$ implies that 
\begin{equation}
\|\varphi _\varepsilon(A-\lambda)(1-\psi_{R/2}) Q^0 \| \le C_s(h/R)^s.
\label{eqn-2.23}
\end{equation}

Then we can make an induction step by $n$ and to prove (\ref{eqn-2.15}). After this, let us replace in (\ref{eqn-2.15}) $R$ and $T$ by $r$ and $t$. Next, for given $R,T$ such that $R\le \varepsilon^3T$ let us plug into (\ref{eqn-2.15}) $n= (T/R)^{\frac{1}{3}}$, 
$t = T^{\frac{2}{3}}R^{\frac{1}{3}}=T/n$ and $r= T^{\frac{1}{3}}R^{\frac{2}{3}}= n R$ (obviously $r\le \varepsilon t$). We arrive to (\ref{eqn-2.14}) with a different but still arbitrarily large exponent $s$.
\end{proof}

As mentioned, we need the following proposition:

\begin{proposition}\label{prop-2.4}
Let conditions of Theorem~\ref{thm-1.1} be fulfilled. Let $\bar{x} \in \bR^d\setminus 0$, $\bar{\xi} \in \cW$ and assume that $0\notin \bar{x}+ \pi_x K^\mp (\xi)$. Let $\cK^\mp$ be a conical $\eta$-vicinity of 
$K^\mp (\xi)$ and $\cV$ be $\eta R$-vicinity of $\bar{x}$, $R=|x|$. Then 
\begin{equation}
\| Q' e^{\pm i h^{-1}T A} Q\| \le C_s (h/R)^s 
\label{eqn-2.24}
\end{equation}
provided $Q=Q(x,hD)$ and $Q'=Q'(x,hD)$ are operators with the symbols  satisfying 
\begin{equation}
|D^\alpha _\xi D^\beta_x Q|\le c_{\alpha\beta} r^{-|\beta|} 
\label{eqn-2.25}
\end{equation}
with $r =T+R$ and $r=R$ respectively, $R\ge R(\eta)$ and support  of symbol of $Q$ does not intersect with 
$\cV  +  \cK^\mp  |_{t=T}$, symbol of $Q'$ is supported in the sufficiently small vicinity of $(\bar{x},\bar{\xi})$.
\end{proposition}

\begin{proof}
Considering propagation in the zone $\{x\colon |x|\asymp r\}$, we see that the propagation speed with respect to 
$\xi$ does not exceed $Cr^{-1-\delta}$. To prove this we scale  $x\mapsto xr^{-1}$, $t\mapsto tr ^{-1}$, $h\mapsto \hbar= hr^{-1}$ and apply the standard energy method (see Chapter~\ref{monsterbook-sect-2} of \cite{IvrMB}). We leave the easy details to the reader. 

Therefore for time $t\asymp r$ variation  of $\xi$ does not exceed $Cr^{-\delta}$.
 Then, the propagation speed  with respect to $\langle x,\ell(\bar{\xi})\rangle $ (which increases) is of magnitude $1$
(as long as $\xi$ remains in the small vicinity of $\bar{\xi}$). Again, to prove it we scale and apply the energy method (see Chapter~\ref{monsterbook-sect-2} of \cite{IvrMB}).

But then the contribution of the time interval $t \asymp r$ to the variation of $\xi$ does not exceed $C r^{-\delta}$ and therefore the variation of $\xi$ for a time interval $[0,T]$ with $T\ge 0$ does not exceed $C R^{-\delta}\le \eta$ for $R\ge R(\eta)$.
\end{proof}

\begin{proposition}\label{prop-2.5}
In the framework of Theorem~\ref{thm-1.1}
\begin{equation}
\| \psi_R  e^{ih^{-1}TA} \varphi _\varepsilon(A-\lambda) \psi_R \|  \le C_s  h^s,
\label{eqn-2.26}
\end{equation}
provided $T\ge C_0 R $, $R\ge 1$. 
\end{proposition}

\begin{proof}
It follows immediately from Proposition~\ref{prop-2.4}  with the semiclassical parameter  $hr^{-1}$ and with $r$ set to its minimal value along the cone of propagation, which is $1$.
\end{proof}

Combining Propositions~\ref{prop-2.3} and~\ref{prop-2.5} we arrive to

\begin{corollary}\label{cor-2.6}
In the framework of Theorem~\ref{thm-1.1}
\begin{equation}
\| \psi_R  e^{ih^{-1}TA} \varphi _\varepsilon(A-\lambda) \psi_R \|  \le C_s  h^s R^s T^{-s}
\label{eqn-2.27}
\end{equation}
provided $T\ge C_0 R $, $R\ge 1$. 
\end{corollary}

\section{Traces and the end of the proof}
\label{sect-2.3}

\begin{proposition}\label{prop-2.7}
In the framework of Theorem~\ref{thm-1.1} the following estimates hold for $T\ge 1$
\begin{align}
&|F_{t\to \tau}  \chi_T(t)  u (x,x,t) |\le C_s h^s (|x|+1) ^{s+1} ((|x|+1)+T)^{-s},
\label{eqn-2.28}\\
&|F_{t\to \tau}  \chi_T(t) \int \psi_R(x) u (x,x,t) \,dx|\le C_s h^s  T^{-s},
\label{eqn-2.29}\\
\intertext{provided $\psi \in \sC_0^\infty (B(0,1))$. Further,}
&|F_{t\to \tau}  \chi_T(t) \int \psi_R(x) u (x,x,t) \,dx|\le C_s h^s R^{-s} T^{-s},
\label{eqn-2.30}
\end{align}
provided $\psi \in \sC_0^\infty (B(0,1)\setminus B(0,\frac{1}{2})$.
\end{proposition}

\begin{proof}
Estimate (\ref{eqn-2.28}) follows immediately from (\ref{eqn-2.26}). Estimate (\ref{eqn-2.30}) follows from (\ref{eqn-2.26}) and
\begin{equation*} 
|F_{t\to \tau}  \chi_T(t) \int \psi_R(x) u (x,x,t)\,dx |\le C_s h^s R^{-s} T,
\end{equation*}
which holds because  we can chose the time direction on the  partition element (see Chapter~\ref{monsterbook-sect-4} of \cite{IvrMB}) and we chose the one in which $|x|\gtrsim R$ (which is possible; see the part of proof of Proposition~\ref{prop-2.3} dealing with $Q^\pm$ and $Q^0$). 

Finally, estimate (\ref{eqn-2.29}) follows from (\ref{eqn-2.30}).
\end{proof}

Then we immediately arrive to the following theorem, which in turn implies Theorem~\ref{thm-1.1}:

\begin{theorem}\label{thm-2.8}
In the framework of Theorem~\ref{thm-1.1} the following estimates hold
\begin{multline}
|\partial_\tau ^k \Bigl( e(x,x,\tau)-\sum _{n\le N-1} \kappa _n (x,\tau) h^{-d+n}\Bigr) |\le\\
C_N h^{-d+N} (|x|+1) ^{-N}   + C_s h^s (|x|+1) ^k
\label{eqn-2.31}
\end{multline}
and for $\psi\in \sC_0^\infty (B(0,1))$
\begin{equation}
|\partial_\tau ^k \int \Bigl( e(x,x,\tau)-\sum _{n\le N-1} \kappa _n (x,\tau) h^{-d+n}\Bigr)\psi_R (x) \,dx|\le
C_N h^{-d+N} .
\label{eqn-2.32}
\end{equation}                 
\end{theorem}

\section{Discussion}
\label{sect-2.4}

\begin{corollary}\label{cor-2.9}
Let conditions of Theorem~\ref{thm-1.1} be fulfilled. Assume that 
\begin{equation}
|D^\alpha _x V|\le c_\alpha (|\xi|+1)^m  (|x|+1)^{-d-|\alpha|-\delta}.
\label{eqn-2.33}
\end{equation}
Then the asymptotics of the Birman-Schwinger spectral shift function 
\begin{equation}
\N_h(\tau)\coloneqq \int \Bigl(e_h (x,x,\tau) -e^0_h (x,x,\tau)\Bigr)\,dx \sim \sum_{n\ge 0} \varkappa h^{-d+n}
\label{eqn-2.34}
\end{equation}
is infinitely differentiable with respect to $\tau$. Here $e^0_h(x,x,\tau)=\kappa_0^0 h^{-d}$ and $e^0_h(x,y,\tau)$ is the Schwartz kernel of spectral projector for $A^0(hD)$, and
\begin{equation}
\varkappa_n (\tau)=\int \bigl(\kappa_n(x,\tau)-\updelta_{n0}\kappa^0 \bigr)\,dx.
\label{eqn-2.35}
\end{equation}
\end{corollary}

Indeed, condition (\ref{eqn-2.33}) guarantees the absolute convergence of integrals in (\ref{eqn-2.35}).

\begin{remark}\label{rem-2.10}
Our results could be easily generalized to non-semi-bounded elliptic $A^0$ (like in Subsection~\ref{IDS-sect-3.1} of \cite{IvrIDS}). Then instead of $e(x,y,\lambda)$ one needs to consider $e(x,y,\lambda,\lambda')$ the Schwartz kernel of $\uptheta (\lambda - A)-\uptheta (\lambda' - A)$
and \underline{either} impose conditions for both $\lambda$ and $\lambda'$, \underline{or} only for $\lambda$ and mollify with respect to~$\lambda'$. 
\end{remark}

It looks strange that the last term in the remainder estimate (\ref{eqn-2.31}) increases as $|x|$ increases, but  
so far I can neither improve it to the uniform with respect to $x$ in the general case, nor show by the counter-example that such improvement is impossible. However I hope to prove 

\begin{conjecture}\label{conj-2.11}
Assume in addition that $A^0$ is a scalar operator and  $\Sigma^0_\lambda =\{\xi\colon A^0(\xi)=\lambda\}$ is  a \emph{strongly convex surface\/} i.e.
\begin{equation}
\pm \sum_{j,k} A^0_{\xi_j\xi_k}(\xi)\eta_j\eta_k \ge \epsilon|\eta|^2\qquad
 \forall \xi\in\Sigma^0_\lambda\ \ \forall \eta\colon \sum_j  A^0_{\xi_j}(\xi)\eta_j=0,
\label{eqn-2.36}
\end{equation}
where the sign depends on the connected component of $\Sigma_\lambda$, containing $\xi$.

Then  the last term in the right-hand expression of \textup{(\ref{eqn-2.31})} could be replaced by 
$C_s h^{s} (|x|+1) ^{k-d}$.
\end{conjecture}

Rationale here is that only few Hamiltonian trajectories from $x$ with $|x|=R\gg 1$ pass close to the origin and even they do not spend much time there.

\end{document}